\documentclass[amstex,12pt]{amsart}
\usepackage{amssymb}
\usepackage{amsmath}
\usepackage{amscd}
\usepackage{mathrsfs}
\usepackage{graphicx}
\usepackage{latexsym}

\theoremstyle{definition} 

\newtheorem*{remark}{Remark}

\theoremstyle{plain}  

\newtheorem{theorem}{Theorem}[section]
\newtheorem{proposition}[theorem]{Proposition}
\newtheorem{corollary}[theorem]{Corollary}
\newtheorem{lemma}[theorem]{Lemma}
\newtheorem*{claim}{Claim}

\newtheorem*{thm}{Theorem}
\newtheorem*{prop}{Proposition}

\DeclareMathOperator{\id}{id}

\DeclareMathOperator{\QC}{QC}

\DeclareMathOperator{\QS}{QS}

\DeclareMathOperator{\Mob}{\mbox{\rm{M\"ob}}}
\DeclareMathOperator{\Diff}{Diff}

\DeclareMathOperator{\Homeo}{Homeo}

\DeclareMathOperator{\Bel}{Bel}

\DeclareMathOperator{\D}{\mathbb D}

\DeclareMathOperator{\R}{\mathbb R}

\DeclareMathOperator{\S1}{\mathbb S}

\begin{document}

\title[Barycentric extension of circle diffeomorphisms]
{Continuity of the barycentric extension of circle diffeomorphisms of H\"older continuous derivatives}

\author{Katsuhiko Matsuzaki}
\address{Department of Mathematics, School of Education, Waseda University,\endgraf
Shinjuku, Tokyo 169-8050, Japan}
\email{matsuzak@waseda.jp}

\makeatletter
\@namedef{subjclassname@2010}{%
\textup{2010} Mathematics Subject Classification}
\makeatother

\thanks{This work was supported by JSPS KAKENHI 25287021.}

\begin{abstract}
The barycentric extension due to Douady and Earle gives a conformally natural 
extension of a quasisymmetric automorphism of the circle
to a quasiconformal automorphism of the unit disk. We consider such extensions for
circle diffeomorphisms of H\"older continuous derivatives and show that this operation is continuous with respect to
an appropriate topology for the space of the corresponding Beltrami coefficients.
\end{abstract}

\maketitle

\section{Introduction}\label{1}

The barycentric extension due to Douady and Earle \cite{DE} gives a natural extension of a self-homeomorphism of
the unit circle $\S1$ to a self-homeomorphism of the unit disk $\D$. It plays an important role applied to
quasisymmetric homeomorphisms of $\S1$ in the complex analytic theory of Teich\-m\"ul\-ler spaces.
In this paper, we apply the barycentric extension to diffeomorphisms of $\S1$ with H\"older continuous derivatives
and obtain an analogous result for the Teich\-m\"ul\-ler space of such circle diffeomorphisms with the universal Teich\-m\"ul\-ler space.

The universal Teich\-m\"ul\-ler space $T$ can be defined as the space $\QS_*(\S1)$ of all normalized quasisymmetric homeomorphisms of $\S1$.
In this setting, the Teich\-m\"ul\-ler projection $q$ is regarded as the boundary extension map on the space $\QC_*(\D)$ of
all normalized quasiconformal homeomorphisms of $\D$. By the measurable Riemann mapping theorem, we can identify the latter space with the space of Beltrami coefficients $\Bel(\D)=L^\infty(\D)_1$, 
which is the open unit ball of measurable functions on $\D$
with the supremum norm. Then $q:\Bel(\D) \to T$ is continuous with respect to the topology on $\QS_*(\S1)$ induced by the
quasisymmetry constant. The barycentric extension yields a continuous section $e:T \to \Bel(\D)$ for $q$.

The Teich\-m\"ul\-ler space $T^\alpha_0$ of circle diffeomorphisms with $\alpha$-H\"older continuous derivatives for $\alpha \in (0,1)$
is similarly defined as a subspace of $T$; the subgroup $\Diff^{1+\alpha}_*(\S1) \subset \QS_*(\S1)$
of all such diffeo\-morphisms with normalization can be defined to be $T^\alpha_0$.
The topology on this group is induced by the $C^{1+\alpha}$-distance from the identity map.
On the other hand, the corresponding subspace of Beltrami coefficients is $\Bel^\alpha_0(\D) \subset \Bel(\D)$, which consists of
all $\mu \in \Bel(\D)$ with finite weighted supremum norm
$$
\Vert \mu \Vert_{\infty,\alpha}={\rm ess.}\sup_{\zeta \in \D}\, \left(\frac{2}{1-|\zeta|^2}\right)^\alpha\,|\mu(\zeta)|. 
$$
Then we have proved in \cite{Mat} that the restriction of the Teich\-m\"ul\-ler projection to $\Bel^\alpha_0(\D)$ gives
a continuous map $q:\Bel^\alpha_0(\D) \to T^\alpha_0$. In fact, the topology of $T^\alpha_0$ coincides with
the quotient topology induced from $\Bel^\alpha_0(\D)$ by $q$. Moreover, a complex Banach manifold structure has been
provided for $T^\alpha_0$ through the Bers embedding. See survey articles \cite{Mat0} for the introduction of
the Teich\-m\"ul\-ler space $T^\alpha_0$ and \cite{Mat1} for applications of $T^\alpha_0$ to problems on circle diffeomorphism
groups.

The main theorem of this paper asserts the continuity of the section $e$ restricted to $T_0^\alpha$.

\begin{theorem}\label{main1}
The barycentric extension of circle diffeomorphisms with $\alpha$-H\"older continuous derivatives gives a {continuous} section
$$
e:T_0^\alpha=\Diff^{1+\alpha}_*(\S1) \to \Bel_0^{\alpha}(\D)
$$
for the Teich\-m\"ul\-ler projection $q$.
\end{theorem}

As a well-known consequence from the existence of a continuous section, we understand
a topological structure of this space. Note that $T_0^\alpha=\Diff^{1+\alpha}_*(\S1)$ is also a topological group \cite{Mat}.

\begin{corollary}
The Teich\-m\"ul\-ler space $T_0^\alpha$ is contractible.
\end{corollary}


In the next section, we will explain the above mentioned concepts and results in more detail.

\medskip
\section{Preliminaries}\label{2}

In this section, we summarize several results on the background of our arguments.
This includes the definition and properties of 
the barycentric extension of quasisymmetric self-homeomorphisms of the circle,
fundamental results on the universal Teich\-m\"ul\-ler space and preliminaries on
the space of circle diffeomorphisms with H\"older continuous derivatives.
For the results mentioned in this section on quasiconformal and quasisymmetric homeomorphisms
as well as Teich\-m\"ul\-ler spaces, we can consult the monograph by Lehto \cite{Leh}.

\subsection{Quasiconformal and quasisymmetric homeomorphisms}
We denote the group of all quasiconformal self-homeomorphisms of the unit disk $\D$ by $\QC(\D)$ and
the group of all quasisymmetric self-homeomorphism of the unit circle $\S1$ by $\QS(\S1)$.
Every $f \in \QC(\D)$ extends continuously to a quasisymmetric homeomorphism of $\S1$. This boundary extension defines
a homomorphism 
$q:\QC(\D) \to \QS(\S1)$.
Conversely, every $\varphi \in \QS(\S1)$ extends continuously to a quasiconformal homeomorphism of $\D$, 
in other words, $q$ is surjective.
In fact, there are explicit ways of giving such quasiconformal extension which defines a section $e:\QS(\S1) \to \QC(\D)$
with $q \circ e={\rm id}|_{\QS}$. The Beurling-Ahlfors extension \cite{BA} and
the Douady-Earle extension \cite{DE} are well-known.

\subsection{The barycentric extension}
The {barycentric extension} or the Douady-Earle extension $e(\varphi)$ of
an orientation-preserving self-homeomorphism $\varphi \in \Homeo(\S1)$ is given as follows.
The {\it average} of $\varphi$ taken at $w \in \D$ is defined by
$$
\xi_\varphi(w)=\frac{1}{2\pi} \int_{\S1}\gamma_w(\varphi(\zeta))|d\zeta|
=\frac{1}{2\pi} \int_{\S1} \frac{\varphi(\zeta)-w}{1-\bar w \varphi(\zeta)}|d\zeta|,
$$
where the M\"obius transformation 
$$
\gamma_w(z)=\frac{z-w}{1-\bar w z} \in \Mob(\D)
$$
sends $w$ to the origin $0$.
The {\it barycenter} of $\varphi$ is a point $w_0 \in \D$ such that $\xi_\varphi(w_0)=0$. This exists uniquely.
The value of the barycentric extension $e(\varphi)$ at the origin $0$ is defined to be the barycenter $w_0$; we set
$e(\varphi)(0)=w_0$.
 
For an arbitrary point $z \in \D$, the barycentric extension $e(\varphi)$ is defined by
$$
{e(\varphi)(z)=e(\varphi \circ \gamma)(0)},
$$
where $\gamma \in \Mob(\D)$ is any M\"obius transformation that maps $0$ to $z$, say, $\gamma=\gamma_z^{-1}$.
This is well-defined since $\xi_{\varphi \circ r}(0)=\xi_\varphi(0)$ for any rotation $r$, which is 
a M\"obius transformation fixing $0$.

An alternative definition was introduced by Lecko and Partyka \cite{LP}.
For each $w \in \D$, we consider the harmonic extension (the Poisson integral) of $\gamma_w \circ \varphi \in \Homeo(\S1)$;
$$
P_w(z):=\frac{1}{2\pi} \int_{\S1} \gamma_w \circ \varphi(\zeta) |\gamma'_z(\zeta)||d\zeta|.
$$
Since $P_w$ is a self-homeomorphism of $\D$ by the Rad\'o-Kneser-Choquet theorem, there exists a unique point $z \in \D$ such that $P_w(z)=0$.
We define a map $e_*(\varphi):\D \to \D$ by $e_*(\varphi)(w)=z$. 
Then $e(\varphi)=e_*(\varphi)^{-1}$. Indeed, $e(\varphi)(z)=w$ and $e_*(\varphi)(w)=z$ are equivalent to
the conditions 
$$
\frac{1}{2\pi} \int_{\S1} \gamma_w \circ \varphi(\gamma_z^{-1}(\tilde \zeta)) |d\tilde \zeta|=0;\quad
\frac{1}{2\pi} \int_{\S1} \gamma_w \circ \varphi(\zeta) |\gamma'_z(\zeta)||d\zeta|=0,
$$
respectively. By substitution $\tilde \zeta=\gamma_z(\zeta)$, we see that these integrals are the same.

The application of the barycentric extension to a quasisymmetric homeomorphism yields the following
fundamental result.

\begin{thm}[\cite{DE}]
For every $\varphi \in \QS(\S1)$, the barycentric extension gives
$e(\varphi) \in \QC(\D)$.
\end{thm}

Besides Douady and Earle \cite{DE}, we can find an expository on the barycentric extension in Pommerenke \cite[Section 5.5]{P},
which we consult occasionally hereafter.

\subsection{Conformal naturality}
The barycentric extension $e(\varphi)$ of $\varphi \in \Homeo(\S1)$ has the conformal naturality in the following sense:
$$
{e(g \circ \varphi \circ \gamma)=g \circ e(\varphi) \circ \gamma}
$$
for any $g, \gamma \in \Mob(\S1)=\Mob(\D)$. Indeed, $e(\varphi \circ \gamma)=e(\varphi) \circ \gamma$
comes from the above definition of $e(\varphi)$. On the other hand, $e(g \circ \varphi)=g \circ e(\varphi)$ comes from a formula
$$
\frac{g(z)-g(w)}{1-\overline{g(w)}g(z)}=e^{i \theta(w)} \frac{z-w}{1-\bar w z}
$$
for some function $\theta:\D \to \R$ of $w$ independent of $z$.
Actually, if $\xi_{\varphi}(w_0)=0$, then
$$
\xi_{g \circ \varphi}(g(w_0))=\frac{1}{2\pi}\int_{\S1}\frac{g(\varphi(\zeta))-g(w_0)}{1-\overline{g(w_0)}g(\varphi(\zeta))}|d\zeta|
=\frac{e^{i\theta(w_0)}}{2\pi}\int_{\S1}\frac{\varphi(\zeta)-w_0}{1-\overline{w_0}\varphi(\zeta)}|d\zeta|=0.
$$

For $f \in \QC(\D)$, we denote the complex dilatation of $f$ by $\mu_f(z)=\bar \partial f(z)/\partial f(z)$.
The conformal naturality of the barycentric extension for quasisymmetric homeomorphisms
in terms of complex dilatations can be described as follows:
$$
\mu_{e(g \circ \varphi \circ \gamma)}(z)=\mu_{g \circ e(\varphi) \circ \gamma}(z)
=\mu_{e(\varphi)}(\gamma(z)) \frac{\overline{\gamma'(z)}}{\gamma'(z)}
$$
for any $g,\gamma \in \Mob(\S1)=\Mob(\D)$ and for any $\varphi \in \QS(\S1)$.
In particular, this implies
$$
{|\mu_{e(g \circ \varphi \circ \gamma)}(z)|=|\mu_{e(\varphi)}(\gamma(z))|}.
$$

\subsection{Continuity of the barycentric extension}\label{2.4}
The subgroups consisting of the {normalized elements} of $\QC(\D)$ and $\QS(\S1)$ fixing three points on $\S1$, say
$1,i,-1$,  
are denoted by ${\QC_*(\D)}$ and ${\QS_*(\S1)}$, respectively.

By the solution of Beltrami equation (the measurable Riemann mapping theorem), $\QC_*(\D)$ is identified with
the space of Beltrami coefficients on $\D$:
$$
\Bel(\D)=\{\mu \in L^\infty(\D) \mid \Vert \mu \Vert_\infty<1\}.
$$
On the other hand, {$\QS_*(\S1)$} can be regarded as the {\it universal Teich\-m\"ul\-ler space} $T$, which is
equipped with the right uniform topology induced by the quasisymmetry constant $M(\varphi) \geq 1$ for $\varphi \in \QS(\S1)$;
a sequence $\varphi_n$ converges to $\varphi$ in $\QS(\S1)$ 
if $M(\varphi_n \circ \varphi^{-1}) \to 1$ $(n \to \infty)$. We note that there are several different ways of defining 
the quasisymmetry constant $M$, say, using the cross ratio, but they all induce the same topology. 

Under the above identification, the restriction of $q$ to $\QC_*(\D)=\Bel(\D)$ plays the role of the Teich\-m\"ul\-ler projection.
A basic property of this projection is the following.

\begin{prop}
The Teich\-m\"ul\-ler projection 
$$
q:\Bel(\D)=\QC_*(\D) \to T=\QS_*(\S1)
$$ 
is continuous and open.
\end{prop}

The section for $q$ given by the barycentric extension is also compatible with the topology.

\begin{thm}[\cite{DE}]
The barycentric section  
$$
e:T=\QS_*(\S1) \to \Bel(\D)=\QC_*(\D)
$$
is continuous. In fact, the composition $e \circ q:\Bel(\D) \to \Bel(\D)$ is real analytic.
\end{thm}

\subsection
{Diffeomorphisms with H\"older continuous derivatives}

An orientation-pre\-serving diffeomorphism $\varphi \in \Diff(\S1)$ belongs to the class
{$\Diff^{1+\alpha}(\S1)$} for $\alpha \in (0,1)$ if {its derivative is
$\alpha$-H\"older continuous}. This means that the lift $\widetilde \varphi:\R \to \R$ of $\varphi$ 
given by $\exp(i \widetilde \varphi(x))=\varphi(e^{ix})$ satisfies
$$
|\widetilde \varphi'(x)-\widetilde \varphi'(y)| \leq c|x-y|^\alpha \qquad (x, y \in \R)
$$
for some constant $c \geq 0$.

We provide $\Diff^{1+\alpha}(\S1)$ with the right uniform topology induced by {$C^{1+\alpha}$-distance 
$p_{1+\alpha}(\varphi)$} from $\rm id$ to $\varphi \in \Diff^{1+\alpha}(\S1)$. Here
$$
p_{1+\alpha}(\varphi):=\sup_{\zeta \in \S1}|\varphi(\zeta)-\zeta|+\sup_{x \in \R}|\widetilde \varphi'(x)-1|+\sup_{x,y \in \R} \frac{|\widetilde \varphi'(x)-\widetilde \varphi'(y)|}{|x-y|^\alpha};
$$
and a sequence $\varphi_n$ is defined to converge to $\varphi$ in $\Diff^{1+\alpha}(\S1)$
if $p_{1+\alpha}(\varphi_n \circ \varphi^{-1}) \to 0$ $(n \to \infty)$. 
Note that $\Diff^{1+\alpha}(\S1)$ is a topological group with this topology \cite{Mat}.

\subsection
{Beltrami coefficients corresponding to $\Diff^{1+\alpha}(\S1)$}
For a Beltrami coefficient $\mu \in \Bel(\D)$, we define an $\alpha$-hyperbolic supremum norm $(\alpha \in (0,1))$
by
$$
\Vert \mu \Vert_{\infty,\alpha}={\rm ess.}\sup_{z \in \D}\, {\rho_{\D}^\alpha(z)}\,|\mu(z)|, 
\quad \rho_{\D}(z)=\frac{2}{1-|z|^2}.
$$
The space of Beltrami coefficients with $\Vert \mu \Vert_{\infty,\alpha}<\infty$ is denoted by 
{$\Bel_0^{\alpha}(\D)$}.

We can characterize $\Diff^{1+\alpha}(\S1)$ by their quasiconformal extension to $\D$.

\begin{thm}
A quasisymmetric homeo\-morphism $\varphi:\S1 \to \S1$ belongs to $\Diff^{1+\alpha}(\S1)$ if and only if
it has a quasiconformal extension $f:\D \to \D$ whose complex dilatation $\mu_f$ belongs to $\Bel_0^{\alpha}(\D)$.
\end{thm}

``Only if'' part was proved by Carleson \cite{Car} using the Beurling-Ahlfors extension of
quasisymmetric functions on the real line.
``If'' part was investigated by Anderson and Hinkkanen \cite{AH} among others,
and settled by Dyn'kin \cite{D} 
and Anderson, Cant\'on and Fern\'andez \cite{ACF}.
A different proof for an improved statement which is necessary to the arguments of Teichm\"uller spaces (Section \ref{2.7})
was given in \cite{Mat}.

\subsection
{The Teich\-m\"ul\-ler space for $\Diff^{1+\alpha}(\S1)$}\label{2.7}

The previous theorem implies that the Teich\-m\"ul\-ler projection (boundary extension) gives a
surjective map
$$ 
q:\Bel_0^{\alpha}(\D) \to \Diff^{1+\alpha}_*(\S1),
$$
where the group $\Diff^{1+\alpha}_*(\S1)$ of the normalized elements can be defined to be the Teich\-m\"ul\-ler space $T_0^\alpha$ of
circle diffeomorphisms with $\alpha$-H\"older continuous derivatives.
Moreover, taking the topology into account, we have proved the following.

\begin{thm}[\cite{Mat}]
The Teich\-m\"ul\-ler projection 
$$
q:\Bel_0^{\alpha}(\D) \to T^\alpha_0=\Diff^{1+\alpha}_*(\S1)
$$ 
is continuous and open.
\end{thm}

Concerning the section given by the barycentric extension, we have also obtained that it has the right image.

\begin{prop}[\cite{Mat}]
The image of the barycentric extension of circle diffeomorphisms with $\alpha$-H\"older continuous derivatives 
$$
e:T_0^\alpha=\Diff^{1+\alpha}_*(\S1) \to \Bel(\D)
$$
is contained in $\Bel_0^{\alpha}(\D)$.
\end{prop}

\medskip
\section{An outline of the proof}\label{3}
\renewcommand{\thetheorem}{\thesection.\arabic{theorem}}

This section is devoted to a sketch of the proof of our main theorem (Theorem \ref{main1}).
The arguments for the rigorous proof begins from the next section.
Since the proof is rather technical and complicated, it will be helpful to mention its outline
before.

We first give a set-up for the proof. Assuming the results in Section \ref{2.7},
we have only to prove the continuity of the barycentric extension $e$ as in the following statement.

\begin{theorem}\label{main2}
Suppose that {$\psi$ converge to $\id$ in $\Diff^{1+\alpha}(\S1)$}.
Then, for every $\varphi_0 \in \Diff^{1+\alpha}(\S1)$,
the complex dilatations {$\mu_{e(\psi \circ \varphi_0)}$} of their barycentric extensions
{converge to $\mu_{e(\varphi_0)}$ in $\Bel^{\alpha}_0(\D)$}, that is,
$$
\sup_{z \in \D}\left(\frac{2}{1-|z|^2}\right)^{\alpha}
|\mu_{e(\psi \circ \varphi_0)}(z)-\mu_{e(\varphi_0)}(z)| \to 0 \qquad (\psi \to \id). 
$$
\end{theorem}

If $e(\psi \circ \varphi_0)=e(\psi)\circ e(\varphi_0)$, the proof would be easy. But, the barycentric extension $e$ is not
a homomorphism; it only has the conformal naturality.
We reduce the theorem to a simpler form by using the following facts:
\begin{enumerate}
\item
Composition of a rotation does not change the derivatives of circle diffeo\-morphisms;
\item
Post-composition of a M\"obius transformation does not change the complex dilatations of quasiconformal
homeomorphisms.
\end{enumerate}
Then we can normalize the situation so that $\varphi_0$ and $\psi$ fix $1$ and the derivative of $\psi$ at $1$ is $1$,
and we 
will estimate the complex dilatations on the real interval $[0,1) \subset \D$. Moreover, we 
have only to consider the convergence when $|z|$ is sufficiently close to $1$. Otherwise,
$2/(1-|z|^2)$ is bounded and 
the uniform convergence of
complex dilatations follows from the convergence $\psi \to \id$ by the arguments for the theorem in Section \ref{2.4}.
Thus the above theorem is reduced to the claim below.
The precise statement respecting the uniformity under conjugations by rotations will be given in Theorem \ref{main3} of Section \ref{6}.

Hereafter, we use the following notation.
Taking the lift $\widetilde \varphi:\R \to \R$ of $\varphi \in \Diff(\S1)$, we define its derivative 
along $\S1$ at $\zeta=e^{ix}$ ($-\pi < x \leq \pi$)
by $\varphi'_{\S1}(\zeta):=\widetilde \varphi'(x)$ .
The distance $d_{\S1}(\zeta,1)$ between $\zeta$ and $1$ along $\S1$ is then $|x|$.
The $\alpha$-H\"older constant of $\psi$ at $1$ is given by
$$
c_\alpha(\psi)(1)=\sup_{1 \neq \zeta \in \S1} \frac{|\psi'_{\S1}(\zeta)-\psi'_{\S1}(1)|}{d_{\S1}(\zeta,1)^\alpha}. 
$$ 

\begin{claim}\label{main3}
Assume that $\psi(1)=\varphi_0(1)=1$ and $\psi'_{\S1}(1)=1$.
If $c_\alpha(\psi)(1)$ converge 
to $0$, then  
$$
\sup_{t_0 \leq t<1}{ \left(\frac{2}{1-t^2}\right)^{\alpha}}
{|\mu_{e(\psi \circ \varphi_0)}(t)-\mu_{e(\varphi_0)}(t)|}\to 0  
$$
for some $t_0<1$ sufficiently close to $1$.
\end{claim}

The strategy for the proof is to use the conjugate by
$$
{h_t(z)=\frac{z+t}{1+tz}} \in \Mob(\D) \qquad (-1<t<1),
$$
which maps the real interval $[-1,1]$ onto itself with the end points fixed and sends $0$ to $t$.
Then the conformal naturality of the barycentric extension implies that
\begin{align*}
\mu_{e(\varphi_0)}(t)&=\mu_{e(h_t^{-1} \circ \varphi_0 \circ h_t)}(0)\frac{\overline{(h_t^{-1})'(0)}}{(h_t^{-1})'(0)};\\
\mu_{e(\psi \circ \varphi_0)}(t)&=\mu_{e(h_t^{-1} \circ \psi \circ \varphi_0 \circ h_t)}(0)\frac{\overline{(h_t^{-1})'(0)}}{(h_t^{-1})'(0)}.
\end{align*}
From these equalities, the term in the above claim we are going to estimate becomes
$$
|\mu_{e(\psi \circ \varphi_0)}(t)-\mu_{e(\varphi_0)}(t)|
=|\mu_{e(h_t^{-1} \circ \psi \circ \varphi_0 \circ h_t)}(0)-\mu_{e(h_t^{-1} \circ \varphi_0 \circ h_t)}(0)|.
$$

The advantage of this reduction is that we can explicitly represent 
$\mu_{e(\varphi)}(0)$ for $\varphi \in \QS(\S1)$ by using the Fourier coefficients for $\varphi$
(including the average of $-\varphi^2$)
if $e(\varphi)(0)=0$, that is, if
$$
a_0:=\xi_{\varphi}(0)=\frac{1}{2\pi}\int_{\S1}\varphi(\zeta)|d\zeta| =0.
$$
Under this condition, we have
$$
\mu_{e(\varphi)}(0)=\frac{a_{-1}-\overline{a_1}b}{a_1-\overline {a_{-1}}b}\ ;
$$
$$
a_1:=\frac{1}{2\pi}\int_{\S1} \bar \zeta \varphi(\zeta)\,|d\zeta|,\quad
a_{-1}:=\frac{1}{2\pi}\int_{\S1} \zeta \varphi(\zeta)\,|d\zeta|,\quad
b:=\frac{-1}{2\pi}\int_{\S1} \varphi(\zeta)^2\,|d\zeta|.
$$
This follows from \cite[p.28]{DE}. See also \cite[p.115]{P}.

However, there are also the following problems in these arguments:
\begin{enumerate}
\item
How can we deal with the weight $(2/(1-t^2))^{\alpha}$ when $t \to 1$.
\item
How can we estimate $\mu_{e(\varphi)}(0)$ even if { $e(\varphi)(0) \neq 0$}; the barycenters of
$h_t^{-1} \circ \varphi_0 \circ h_t$ and $h_t^{-1} \circ \psi \circ \varphi_0 \circ h_t$ are not necessarily zero.
\end{enumerate}
 
The solution to problem (1) is given by the precision of the following result due to Earle \cite{E}:
If $\psi(1)=1$ and $\psi'_{\S1}(1)=1$ then
$h_t^{-1} \circ \psi \circ h_t$ converge to $\rm id$ uniformly on $\S1$ as $t \to 1$.
This is because the conjugation by $h_t$ magnifies the mapping of $\psi$ near $1$, and since
the linear approximation of $\psi$ has slope $\psi'_{\S1}(1)=1$, it converges to the identity.
Earle gave a more precise statement for it ``with future applications in mind''.
We follow his arguments at the present by utilizing
the {$\alpha$-H\"older constant of $c_\alpha(\psi)(1)$.
Integration of the definition of the $\alpha$-H\"older constant (Proposition \ref{setup}) yields
$$
|\psi(\zeta)-\zeta| \leq C|\zeta-1|^{\alpha+1} \quad (\zeta \in \S1),\quad
C=\frac{(\pi/2)^{\alpha+1}}{\alpha+1} \cdot c_\alpha(\psi)(1).
$$
This can make the above result by Earle
to be a quantitative statement as follows. The proof will be given in Section \ref{4}.

\begin{lemma}\label{key}
Suppose that $\psi \in \Homeo(\S1)$ satisfies 
$$
|\psi(\zeta)-\zeta| \leq C|\zeta-1|^{\alpha+1} \quad (\zeta \in \S1)
$$
for some constant $C \leq 1/4$.
Set {$\psi_t=h_t^{-1} \circ \psi \circ h_t$}. 
Choose any $\varepsilon>0$. If 
$1-t \leq \frac{1}{4}(\varepsilon/(4C))^{1/\alpha}$,
then
$|\psi_t(\zeta)-\zeta| \leq \varepsilon$
for every $\zeta \in \S1$.
\end{lemma}

This asserts that $\psi_t$ is uniformly close to $\rm id$ {in the order of 
$4^{\alpha+1}C(1-t)^\alpha$}
as $t \to 1$.
Hence this order offsets the problematic weight $(2/(1-t^2))^{\alpha}$ and moreover
the convergence $C \to 0$, which comes from the assumption $c_\alpha(\psi)(1) \to 1$, supports our Theorem \ref{main2}.

Towards the solution to problem (2),
we consider the barycenter $e(\varphi_t)(0)$ of the conjugate {$\varphi_t=h_t^{-1} \circ \varphi_0 \circ h_t$}.
Even if $e(\varphi_t)(0) \neq 0$, we can estimate the Fourier coefficients for $\varphi_t$ uniformly {if
$e(\varphi_t)(0)$ is in a compact subset of $\D$}.

For the base point $\varphi_0 \in \Diff^{1+\alpha}(\S1)$, the derivative $(\varphi_0)'_{\S1}(1)$ is not necessarily $1$.
In this case, the close-up of the behavior of $\varphi_0$ in a neighborhood of $1$ by the conjugation of $h_t$
converges to the M\"obius transformation $h_s$ satisfying $(h_s)'_{\S1}(1)=(\varphi_0)'_{\S1}(1)$. 
More concretely, this is given in the following claim.
The corresponding statement respecting the uniformity under normalization by rotation will be given in Lemma \ref{akari2}.

\begin{claim}
For $(\varphi_0)'_{\S1}(1)=\ell>0$,
take $h_s \in \Mob(\S1)$ with $(1-s)/(1+s)=\ell$.
Then $\varphi_t$ converge uniformly to $h_s$ on $\S1$.
\end{claim}

Fix $t$ sufficiently close to $1$.
Then the claim says that $\varphi_t$ is uniformly close to $h_s$.
Under this condition, we can expect that
the barycenter $e(\varphi_t)(0)$ should be close to
$e(h_s)(0)=s$, which is to be verified in Section \ref{6}. Hence, for some $g_1 \in \Mob(\D)$
close to $h_s$ (written as $g_1 \fallingdotseq h_s$), we will have
$$
e(g_1^{-1} \circ \varphi_t)(0)=0.
$$
Similarly, since $\psi_t=h_t^{-1} \circ \psi \circ h_t$ tends to $\id$ by assumption,
$$
\psi_t \circ \varphi_t=h_t^{-1} \circ \psi \circ \varphi_0 \circ h_t
$$
is close to $h_s$. Hence, for some $g_2\ (\fallingdotseq h_s) \in \Mob(\D)$, 
$$
e(g_2^{-1} \circ \psi_t \circ \varphi_t)(0)=0.
$$

Now we represent the complex dilatations as
\begin{align*}
\mu_{e(\varphi_t)}(0)&=\mu_{e(g_1^{-1} \circ \varphi_t)}(0)=\frac{a_{-1}-\overline{a_1}b}{a_1-\overline {a_{-1}}b};\\
\mu_{e(\psi_t \circ \varphi_t)}(0)&=\mu_{e(g_2^{-1} \circ \psi_t\circ \varphi_t)}(0)=\frac{a'_{-1}-\overline{a'_1}b}{a'_1-\overline {a'_{-1}}b'}.
\end{align*}
Here $a_1, a_{-1}, b$ are the Fourier coefficients for $g_1^{-1} \circ \varphi_t$ and $a'_1, a'_{-1}, b'$ are 
the Fourier coefficients for $g_2^{-1} \circ \psi_t \circ \varphi_t$. 
By using the fact that $g_1 \fallingdotseq g_2$, we can estimate
$$
{|\mu_{e(\psi_t \circ \varphi_t)}(0)-\mu_{e(\varphi_t)}(0)|}
$$
in terms of the approximation of $h_s$ by $g_1$ and $g_2$.
This will be carried out precisely in Section \ref{6}.

\medskip
\section{Convergence of conjugation of circle diffeomorphisms}\label{4}

In this section, we prepare certain results on the convergence of
conjugation of circle diffeomorphisms by the canonical M\"obis transformations,
which is inspired by the paper of Earle \cite{E}. These are necessary for
the proof of our main theorem concerning the solution of the problems mentioned in
the previous section.

In what follows, it is convenient to regard $\S1$ being parametrized by arc length. 
For $\zeta_1, \zeta_2 \in \S1$, the the length of the shorter
circular arc connecting them is denoted by
$d_{\S1}(\zeta_1,\zeta_2)$. 
By the universal cover $\zeta=e^{ix}:\R \to \S1$, this is given by
$$
d_{\S1}(\zeta_1,\zeta_2)=\min\{|x_1-x_2| \mid \zeta_1=e^{ix_1},\ \zeta_2=e^{ix_2}\} \leq \pi.
$$
For $\varphi_1,\varphi_2 \in \Homeo(\S1)$, we set
$$
\Vert \varphi_1- \varphi_2 \Vert_{\S1}=\sup_{\zeta \in \S1}d_{\S1}(\varphi_1(\zeta),\varphi_2(\zeta)).
$$
Define $\widetilde \varphi: \R \to \R$ to be a lift of $\varphi \in \Homeo(\S1)$
with $\exp(i \widetilde \varphi(x))=\varphi(e^{ix})$.
For $\varphi \in \Diff(\S1)$,
its derivative along $\S1$ at $\zeta=e^{ix}$ is defined by $\varphi'_{\S1}(\zeta):=\widetilde \varphi'(x)$.
The $\alpha$-H\"older constant of $\varphi$ at $\eta=e^{iy} \in \S1$ is given by
$$
c_\alpha(\varphi)(\eta)=\sup_{\zeta \in \S1} \frac{|\varphi'_{\S1}(\zeta)-\varphi'_{\S1}(\eta)|}{d_{\S1}(\zeta,\eta)^\alpha}
=\sup_{y \neq x \in \R} \frac{|\widetilde \varphi'(x)-\widetilde \varphi'(y)|}{|x-y|^\alpha}. 
$$ 

First, we prepare an elementary fact on the integration of the $\alpha$-H\"older continuity condition at $1 \in \S1$.

\begin{proposition}\label{setup}
Suppose that $\psi \in \Diff(\S1)$ with $\psi(1)=1$ and $\psi'_{\S1}(1)=1$ satisfies
$$
|\psi'_{\S1}(\zeta)-1| \leq c d_{\S1}(\zeta,1)^\alpha
$$
for some constant $c \geq 0$. Then 
$$
|\psi(\zeta)-\zeta| \leq \frac{c(\pi/2)^{\alpha+1}}{\alpha+1}|\zeta-1|^{\alpha+1}.
$$
\end{proposition}

\begin{proof}
The lift $\widetilde \psi$ with $\widetilde \psi(0)=0$ satisfies
$|\widetilde \psi'(x)-1| \leq c|x|^\alpha$ for $\zeta=e^{ix}$ $(-\pi<x \leq \pi)$. This can be written as
$$
1-c|x|^\alpha \leq \widetilde \psi'(x) \leq 1+c|x|^\alpha.
$$
Then the integration from $0$ to $x$ yields
$$
x-\frac{c}{\alpha+1}|x|^{\alpha+1} \leq \widetilde \psi(x) \leq x+\frac{c}{\alpha+1}|x|^{\alpha+1}.
$$
Hence
$$
|\psi(\zeta)-\zeta| \leq |\widetilde \psi(x)-x| \leq \frac{c}{\alpha+1}|x|^{\alpha+1} 
\leq \frac{c}{\alpha+1}\{(\pi/2)|\zeta-1|\}^{\alpha+1}
$$
for $\zeta=e^{ix}$, which is the required inequality.
\end{proof}

For $t \in (-1,1)$, we utilize a particular M\"obius transformation of $\D$ given by
$$
h_t(z)=\frac{z+t}{1+tz},
$$
which maps the real interval $[-1,1]$ onto itself with the end points fixed and sends $0$ to $t$.
The following lemma, mentioned in Section \ref{3}, is an application of the arguments in Earle \cite[Theorem 2]{E}
to an orientation-preserving self-homeomorphism $\psi \in \Homeo(\S1)$ approximating the identity
with a prescribed order at the fixed point $1 \in \S1$. The conjugate of $\psi$ by $h_t$ expands the local behavior of $\psi$ near $1$
to the global $\S1$.

\newtheorem{keylemma}{Lemma}
\renewcommand{\thekeylemma}{\ref{key}}
\begin{keylemma}
Suppose that $\psi \in \Homeo(\S1)$ satisfies 
$$
|\psi(\zeta)-\zeta| \leq C|\zeta-1|^{\alpha+1}
$$
for some constant $C \leq 1/4$.
Set $\psi_t=h_t^{-1} \circ \psi \circ h_t$. 
Choose any $\varepsilon>0$. If 
$1-t \leq \frac{1}{4}(\varepsilon/(4C))^{1/\alpha}$,
then
$|\psi_t(\zeta)-\zeta| \leq \varepsilon$
for every $\zeta \in \S1$.
\end{keylemma}

\begin{proof}
Set $\omega=h_t(\zeta)$. Then
\begin{align*}
|\psi_t(\zeta)-\zeta|
&=|h_t^{-1}(\psi(\omega))-h_t^{-1}(\omega)|\\
&= \frac{(1-t^2)\,|\psi(\omega)-\omega|}{|1-t\psi(\omega)|\cdot|t\omega-1|}
\leq \frac{2C(1-t)\,|\omega-1|^{\alpha+1}}{|1-t\psi(\omega)|\cdot|t\omega-1|}.
\end{align*}
By using $1-t \leq |1-t\psi(\omega)|$ and $|\omega-1| \leq 2|t\omega-1|$, we have
$$
|\psi_t(\zeta)-\zeta| \leq 4C|\omega-1|^\alpha.
$$
Set $\delta=(\varepsilon/(4C))^{1/\alpha}$. Then $4C|\omega-1|^\alpha \leq \varepsilon$ if 
$|\omega-1| \leq \delta$. Hence we have only to consider the case of $|\omega-1| \geq \delta$.

As before, we have
$$
|\psi_t(\zeta)-\zeta|
\leq \frac{2C(1-t)\,|\omega-1|^{\alpha+1}}{|1-t\psi(\omega)|\cdot|t\omega-1|}
\leq \frac{4C(1-t)\,|\omega-1|^{\alpha}}{|1-t\psi(\omega)|}.
$$ 
In this time, we use $|1-\psi(\omega)| \leq 2|1-t\psi(\omega)|$. Moreover, since $C \leq 1/4$,
\begin{align*}
|1-\psi(\omega)| &\geq |\omega-1|-|\psi(\omega)-\omega|\\
&\geq |\omega-1|(1-C|\omega-1|^\alpha) \geq |\omega-1|/2.
\end{align*}
Plugging these estimates into the above inequality, we conclude
$$
|\psi_t(\zeta)-\zeta| \leq 16C(1-t)|\omega-1|^{\alpha-1}\leq 16C(1-t)\delta^{\alpha-1}.
$$
If $1-t \leq \frac{1}{4}(\varepsilon/(4C))^{1/\alpha}$, then using $\delta=(\varepsilon/(4C))^{1/\alpha}$ we have
$$
16C(1-t)\delta^{\alpha-1} \leq \varepsilon.
$$
This completes the proof of the assertion.
\end{proof}

In the later application, we consider the situation where the constant $c$ in Proposition \ref{setup},
which will be taken as the $\alpha$-H\"older constant $c_\alpha(\psi)(1)$ of $\psi_{\S1}'$ at $1 \in \S1$, can be arbitrarily small.
Then we can choose the constant $C$ in Lemma \ref{key} as 
$$
C=\frac{c(\pi/2)^{\alpha+1}}{\alpha+1} \leq \frac{1}{4},
$$
and apply the consequence of this lemma.

%
%

We denote the rotation sending $1$ to $\eta \in \S1$ by $r_\eta \in \Mob(\S1)$.
The composition of rotations does not change the derivative at any point $\eta \in \S1$ of a diffeomorphism $\varphi_0 \in \Diff^1(\S1)$.
Hence we may assume that it fixes $1$. The previous lemma dealt with the case of its derivative at $1$ is $1$.
The following lemma treats the general case and asserts the convergence of the conjugate by $h_t$ to 
an appropriate M\"obius transformation.

\begin{lemma}\label{akari2}
Let $\varphi_0 \in \Diff^1(\S1)$ and $\eta \in \S1$.
Take rotations $r_{\eta}, r_{\varphi_0(\eta)} \in \Mob(\S1)$ 
and set 
$$
\varphi_0^\eta=r^{-1}_{\varphi_0(\eta)} \circ \varphi_0 \circ r_{\eta},
$$
which fixes $1 \in \S1$.
Let $(\varphi_0^\eta)'_{\S1}(1)=\ell_\eta>0$ and
take $h_{s_\eta} \in \Mob(\S1)$ such that $s_\eta \in (-1,1)$ satisfies $(1-s_\eta)/(1+s_\eta)=\ell_\eta$.
Set 
$$
\varphi_t^\eta=h_t^{-1} \circ \varphi_0^\eta \circ h_t.
$$ 
Then,
for any $\varepsilon_0 \in (0,2]$, there exists $\delta_0>0$ depending only on $\varepsilon_0$ and $\varphi_0$
but not depending on $\eta \in \S1$
such that if $1-t \leq \delta_0$ then
$$
|\varphi_t^\eta(\zeta)-h_{s_\eta}(\zeta)| \leq \varepsilon_0
$$
for every $\zeta \in \S1$ and for every $\eta \in \S1$.
\end{lemma}

\begin{proof}
Set $\omega=h_t(\zeta)$. Then
\begin{align*}
|\varphi_t^\eta(\zeta)-h_{s_\eta}(\zeta)|&=|h_t^{-1}(\varphi_0^\eta(\omega))-h_t^{-1}(h_{s_\eta}(\omega))|
= \frac{(1-t^2)\,|\varphi_0^\eta(\omega)-h_{s_\eta}(\omega)|}{|1-t\varphi_0^\eta(\omega)|\cdot |1-t h_{s_\eta}(\omega)|}.
\end{align*}
We will estimate the difference between $\varphi_0^\eta$ and $h_{s_\eta}$ near $1$. Note that $\varphi_0^\eta(1)=h_{s_\eta}(1)=1$
and $(\varphi_0^\eta)'_{\S1}(1)=(h_{s_\eta})'_{\S1}(1)=\ell_\eta$.

\begin{claim}
For any $\tilde \varepsilon>0$, there exists $\tilde \delta>0$ independent of $\eta$ such that
if $|h_{s_\eta}(\omega)-1| \leq \tilde \delta$ then
$$
|\varphi_0^\eta(\omega)-h_{s_\eta}(\omega)| \leq \tilde \varepsilon |h_{s_\eta}(\omega)-1|.
$$
\end{claim}

\begin{proof}
Take the lift $\widetilde \varphi_0^\eta$ of $\varphi_0^\eta$ with $\widetilde \varphi_0^\eta(0)=0$. Then
$$
\widetilde \varphi_0^\eta(x)=\ell_\eta \, x+\{(\widetilde \varphi_0^\eta)'(\xi)-(\widetilde \varphi_0^\eta)'(0)\}\,x
$$
for some $\xi \in \R$ between $0$ and $x$. Since $(\widetilde \varphi_0^\eta)'$ is uniformly eqi-continuous independent of $\eta$,
$|(\widetilde \varphi_0^\eta)'(\xi)-(\widetilde \varphi_0^\eta)'(0)|$ is bounded by some constant $c(x)>0$ with $c(x) \to 0$
($x \to 0$).
Hence
$$
|\widetilde \varphi_0^\eta(x)-\ell_\eta x| \leq c(x)|x| \qquad (\forall \eta \in \S1).
$$
We consider the same estimate for the lift $\widetilde h_{s_\eta}$ of $h_{s_\eta}$. Since $s_\eta$ is uniformly bounded
away from $-1$ and $1$ (as $\ell_\eta$ is uniformly bounded
away from $0$ and $\infty$)
independent of $\eta$, we also have some constant $c_*(x)>0$ with $c_*(x) \to 0$ ($x \to 0$)
such that
$$
|\widetilde h_{s_\eta}(x)-\ell_\eta x| \leq c_*(x)|x| \qquad (\forall \eta \in \S1).
$$

On the other hand, since $\widetilde h_{s_\eta}(x)=\widetilde h_{s_\eta}'(\xi_*)\,x$ for some $\xi_* \in \R$
and since $\widetilde h_{s_\eta}'(\xi_*) \geq \min\{\ell_\eta,\ell_\eta^{-1}\}$, we have
$$
|x| \leq \frac{1}{\min_{\eta \in \S1} \ell_\eta^{\pm 1}}\,|\widetilde h_{s_\eta}(x)|.
$$
Therefore we obtain that
$$
|\widetilde \varphi_0^\eta(x)-\widetilde h_{s_\eta}(x)| \leq (c(x)+c_*(x))|x| 
\leq \frac{c(x)+c_*(x)}{\min_{\eta \in \S1} \ell_\eta^{\pm 1}}\,|\widetilde h_{s_\eta}(x)|.
$$
Here, $\widetilde h_{s_\eta}(x) \to 0$ implies $x \to 0$ and then 
the coefficient of $|\widetilde h_{s_\eta}(x)|$ in the last term
tends to $0$. Transforming this inequality for $\varphi_0^\eta(\omega)$ and $h_{s_\eta}(\omega)$, we can verify
the required claim.
\end{proof}

\noindent
{\it Proof of Lemma \ref{akari2} continued}.
For a given $\varepsilon_0 \in (0,2]$, set $\tilde \varepsilon=\varepsilon_0/4$ and
choose $\tilde \delta$ as in the claim. First, we consider the case where $|h_{s_\eta}(\omega)-1| \leq \tilde \delta$.
Then, by $|1-t\varphi_0^\eta(\omega)| \geq 1-t$ and $2|1-t h_{s_\eta}(\omega)| \geq |1-h_{s_\eta}(\omega)|$, the claim shows that
$$
\frac{(1-t^2)\,|\varphi_0^\eta(\omega)-h_{s_\eta}(\omega)|}{|1-t\varphi_0^\eta(\omega)|\cdot |1-t h_{s_\eta}(\omega)|}
\leq \frac{2(1-t^2)\,|\varphi_0^\eta(\omega)-h_{s_\eta}(\omega)|}{(1-t)\cdot |h_{s_\eta}(\omega)-1|} \leq 4\tilde \varepsilon=\varepsilon_0.
$$
Thus we obtain $|\varphi_t^\eta(\zeta)-h_{s_\eta}(\zeta)| \leq \varepsilon_0$ without taking care of $t$ in this case.

Next, we consider the case where $|h_{s_\eta}(\omega)-1| \geq \tilde \delta$. Then, using 
$2|1-t\varphi_0^\eta(\omega)| \geq |1-\varphi_0^\eta(\omega)|$ in addition, we have
$$
\frac{(1-t^2)\,|\varphi_0^\eta(\omega)-h_{s_\eta}(\omega)|}{|1-t\varphi_0^\eta(\omega)|\cdot |1-t h_{s_\eta}(\omega)|}
\leq
\frac{4(1-t^2)\,|\varphi_0^\eta(\omega)-h_{s_\eta}(\omega)|}{|1-\varphi_0^\eta(\omega)|\cdot |1-h_{s_\eta}(\omega)|}
\leq
\frac{16(1-t)}{\tilde \delta |1-\varphi_0^\eta(\omega)|}.
$$
Here, if $|h_{s_\eta}(\omega)-1| = \tilde \delta$ then 
\begin{align*}
|1-\varphi_0^\eta(\omega)| &\geq |h_{s_\eta}(\omega)-1|-|h_{s_\eta}(\omega)-\varphi_0^\eta(\omega)|\\
&\geq (1-\tilde \varepsilon)|h_{s_\eta}(\omega)-1| \geq \tilde \delta /2
\end{align*}
by the above claim and $\tilde \varepsilon \leq 1/2$. However, since $\varphi_0^\eta$ is a self-homeomorphism of $\S1$,
this is also true even for $|h_{s_\eta}(\omega)-1| > \tilde \delta$. Hence 
$$
|\varphi_t^\eta(\zeta)-h_{s_\eta}(\zeta)| \leq \frac{16(1-t)}{\tilde \delta|1-\varphi_0^\eta(\omega)|} \leq \frac{32(1-t)}{\tilde \delta^2}.
$$
By choosing $\delta_0=\varepsilon_0 \tilde \delta^2/32$, we obtain the assertion.
\end{proof}

\medskip
\section{Average of circle homeomorphisms}\label{5}

The barycentric extension is defined by considering the average of a circle homeomorphism.
In this section, we will show necessary properties of the average and the vector field given by the average function.

Recall that the M\"obius transformation $\gamma_w \in \Mob(\D)$ is defined by
$$
\gamma_w(z)=\frac{z-w}{1-\bar w z}
$$
for each $w \in \D$. 
First, we list up properties of $\gamma_w$ which will be used later. They are verified easily.

\begin{proposition}\label{claims}
The M\"obius transformation $\gamma_w \in \Mob(\D)$ for each $w \in \D$ satisfies the following:
\begin{enumerate}
\item
$\displaystyle{|\gamma_w(z)-z| \leq \frac{2|w|}{1-|w|}}$ for every $z \in \D$;
\item
$\displaystyle{|\gamma'_w(\zeta)|=\frac{1-|w|^2}{|\zeta-w|^2}}$
is the Poisson kernel, which satisfies
$\displaystyle{\frac{1-|w|}{1+|w|} \leq |\gamma'_w(\zeta)| \leq \frac{1+|w|}{1-|w|}}$ for every $\zeta \in \S1$;
\item
$\displaystyle{\frac{1}{2\pi}\int_{\S1} \gamma_w(\zeta)\,|d\zeta|=-w}$.
\end{enumerate}
\end{proposition}
\smallskip

For $\varphi \in \Homeo(\S1)$, we define its average taken at $w \in \D$ as
$$
\xi_\varphi (w)=\frac{1}{2\pi} \int_{\S1} \frac{\varphi(\zeta)-w}{1-\bar w \varphi(\zeta)}\,|d\zeta|.
$$
Then $\xi_\varphi$ 
is a complex-valued differentiable function on $\D$, 
which can be regarded as 
a vector field on $\D$. If $\varphi \in \Homeo(\S1)$ is close to $\rm id$, then
the vector field $\xi_\varphi$ is close to $\xi_{\rm id}$ as the following claim shows.

\begin{proposition}\label{av-id}
If $\varphi \in \Homeo(\S1)$ satisfies $\Vert \varphi-\id \Vert_{\S1} < \varepsilon$,
then
$|\xi_\varphi(w)-\xi_{\rm id}(w)| < 2\varepsilon$ for every $w \in \D$.
\end{proposition}

\begin{proof}
The definition of $\xi$ implies that
$$
|\xi_\varphi(w)-\xi_{\rm id}(w)|
=\left|\frac{1}{2\pi} \int_{\S1} \gamma_w(\varphi(\zeta))\,|d\zeta|-\frac{1}{2\pi} \int_{\S1} \gamma_w(\zeta)\,|d\zeta| \right|.
$$
Then this is estimated from above by
$$
\frac{1}{2\pi} \int_{\S1} |\gamma_w(\varphi(\zeta))-\gamma_w(\zeta)|\,|d\zeta|
\leq \frac{1}{2\pi} \int_{\S1}\left(\int_\zeta ^{\varphi(\zeta)}|\gamma'_w(\eta)|\,|d\eta| \right)\,|d \zeta|,
$$
where the inner path integral is along the circular arc from $\zeta$ to $\varphi(\zeta)$.
Since $d_{\S1}(\varphi(\zeta),\zeta) < \varepsilon$, this integral is strictly bounded by
$\int^{\zeta+\varepsilon}_{\zeta-\varepsilon} |\gamma'_w(\eta)|\,|d\eta|$. Hence we have
$$
|\xi_\varphi(w)-\xi_{\rm id}(w)|
< \frac{1}{2\pi} \int_{\S1}\left(\int^{\zeta+\varepsilon}_{\zeta-\varepsilon} |\gamma'_w(\eta)|\,|d\eta| \right)\,|d \zeta|
\leq \frac{2\varepsilon}{2\pi} \int_{\S1} |\gamma'_w(\eta)|\,|d\eta|=2\varepsilon.
$$
Here, the last equality is due to the fact that $|\gamma'_w(\eta)|$ is the Poisson kernel by Proposition \ref{claims} (2).
\end{proof}
\medskip

\begin{remark}
Since $\xi_{\rm id}(w)=-w$ by Proposition \ref{claims} (3), we have
$|\xi_\varphi(w)+w|< 2\varepsilon$ in Proposition \ref{av-id}.
\end{remark}

The barycenter of $\varphi \in \Homeo(\S1)$ is defined to be a point $w \in \D$ such that $\xi_\varphi(w)=0$.
It can be shown that it exists uniquely for every $\varphi \in \Homeo(\S1)$ (see \cite[Proposition 1]{DE}, \cite[Lemma 5.20]{P}).

\begin{corollary}\label{barycenter}
If $\varphi \in \Homeo(\S1)$ satisfies $\Vert \varphi-{\rm id}\Vert_{\S1} < \varepsilon$,
then the barycenter $w \in \D$ of $\varphi$ satisfies $|w| < 2\varepsilon$.
\end{corollary}

\begin{proof}
The barycenter $w$ of $\varphi$ satisfies $\xi_\varphi(w)=0$ by definition. Then the result follows from Proposition \ref{av-id}
and the remark after that.
\end{proof}

We generalize the above proposition to an assertion on the difference between any two average functions and 
moreover on the difference between their derivatives.

\begin{proposition}\label{av-any}
For any $\varphi, \psi \in \Homeo(\S1)$, the following inequalities are satisfied for every $w \in \D$:
\begin{enumerate}
\item
$\displaystyle{|\xi_\varphi(w)-\xi_\psi(w)| \leq \frac{1+|w|}{1-|w|} \Vert \varphi-\psi \Vert_{\S1}}$;
\item
$\displaystyle{|\partial \xi_\varphi(w)-\partial \xi_\psi(w)| \leq \frac{|w|}{(1-|w|)^2} \Vert \varphi-\psi \Vert_{\S1}}$;
\item
$\displaystyle{|\bar \partial \xi_\varphi(w)-\bar \partial \xi_\psi(w)| 
\leq \frac{(2-|w|)(1+|w|)^2}{(1-|w|)^4} \Vert \varphi-\psi \Vert_{\S1}}$.
\end{enumerate}
\end{proposition}

\begin{proof}
(1) Simple computation yields
$$
\frac{\varphi(\zeta)-w}{1-\bar w \varphi(w)}-\frac{\psi(\zeta)-w}{1-\bar w \psi(w)}=
\frac{(1-|w|^2)(\varphi(\zeta)-\psi(\zeta))}{(1-\bar w \varphi(\zeta))(1-\bar w \psi(\zeta))}.
$$
Estimating the absolute value of the denominator from below by $(1-|w|)^2$, we have the assertion.

(2) The $\partial$-derivative of $\xi_\varphi$ is 
$$
\partial \xi_\varphi(w)=\frac{1}{2\pi} \int_{\S1} \frac{-1}{1-\bar w \varphi(\zeta)}\,|d\zeta|
$$
and the same is true for $\xi_\psi$. Then
$$
\frac{-1}{1-\bar w \varphi(\zeta)}-\frac{-1}{1-\bar w \psi(\zeta)}
=\frac{-\bar w(\varphi(\zeta)-\psi(\zeta))}{(1-\bar w \varphi(\zeta))(1-\bar w \psi(\zeta))}.
$$
By the same estimate for the denominator as before, we have the assertion.

(3) The $\bar \partial$-derivative of $\xi_\varphi$ is 
$$
\bar \partial \xi_\varphi(w)=\frac{1}{2\pi} \int_{\S1} \frac{(\varphi(\zeta)-w)\varphi(\zeta)}{(1-\bar w \varphi(\zeta))^2}\,|d\zeta|
$$
and the same is true for $\xi_\psi$. Then
\begin{align*}
&\quad \frac{(\varphi(\zeta)-w)\varphi(\zeta)}{(1-\bar w \varphi(\zeta))^2}-\frac{(\psi(\zeta)-w)\psi(\zeta)}{(1-\bar w \psi(\zeta))^2}\\
&=\frac{(\varphi(\zeta)-\psi(\zeta))\{\varphi(\zeta)+\psi(\zeta)+\bar w(|w|^2-2)\varphi(\zeta)\psi(\zeta)-w\}}
{(1-\bar w \varphi(\zeta))^2 (1-\bar w \psi(\zeta))^2}
\end{align*}
Here we estimate the absolute value of a factor of the numerator as
\begin{align*}
&\quad |\varphi(\zeta)+\psi(\zeta)+\bar w(|w|^2-2)\varphi(\zeta)\psi(\zeta)-w|\\
&\leq 2+|w|(2-|w|^2)+|w|
=(2-|w|)(1+|w|)^2.
\end{align*}
By the same estimate for the denominator as before, we have the assertion.
\end{proof}
\medskip

Next, we will see that if $\varphi \in \Homeo(\S1)$ is close to $\rm id$ and normalized so that 
its barycenter is at the origin $0 \in \D$, then $|\xi_\varphi(w)|$ can be estimated from below by
$|\xi_{\rm id}(w)|=|w|$ near the origin.

\begin{lemma}\label{lower}
Suppose that $\varphi \in \Homeo(\S1)$ satisfies $\Vert \varphi-{\rm id} \Vert_{\S1} \leq \varepsilon$
and $\xi_\varphi(0)=0$. Then
$$
(1-56 \varepsilon)|w| \leq |\xi_\varphi(w)|
$$
for every $w \in \D$ with $|w| \leq 1/2$. 
\end{lemma}

\begin{proof}
For any such $w \in \D$, take the segment connecting to $0 \in \D$. 
Represent this segment by $\gamma(s)$ for the arc length parameter $s \in [0,|w|]$
with $\gamma(0)=0$ and $\gamma(|w|)=w$.
Then 
$$
\xi_\varphi(w)=\int_0^{|w|} \frac{d \xi_\varphi(\gamma(s))}{ds}\,ds
=\int_0^{|w|} (\partial \xi_\varphi(\gamma(s))e^{i\theta}+\bar \partial \xi_\varphi(\gamma(s))e^{-i\theta}) ds,
$$
where $\theta=\arg w$. From this, we have
$$
\xi_\varphi(w)+w=e^{i\theta} \int_0^{|w|}(\partial \xi_\varphi(\gamma(s))+1+\bar \partial \xi_\varphi(\gamma(s)) e^{-2i\theta})ds.
$$
For $|w| \leq 1/2$, we apply Proposition \ref{av-any} (2) and (3) to obtain
\begin{align*}
|\xi_\varphi(w)+w| &\leq \int_0^{|w|} |\partial \xi_\varphi(\gamma(s))+1|\,ds
+\int_0^{|w|}|\bar \partial \xi_\varphi(\gamma(s))|\,ds\\
&\leq 2 \varepsilon |w|+54 \varepsilon |w|=56 \varepsilon |w|.
\end{align*}
It follows that $(1-56 \varepsilon)|w| \leq |\xi_\varphi(w)|$, which is the required inequality.
\end{proof}
\medskip

We choose $\varepsilon_0>0$ so that $\varepsilon_0 \leq 1/112$. 
Under this condition, if $\Vert \varphi-{\rm id} \Vert_{\S1} \leq \varepsilon_0$ and $\xi_\varphi(0)=0$, then
$|\xi_\varphi(w)| \geq |w|/2$ for $|w| \leq 1/2$ by Lemma \ref{lower}.

\begin{lemma}\label{rouche}
Assume that $\varphi_0 \in \Homeo(\S1)$ satisfies $\xi_{\varphi_0}(0)=0$ and
$|\xi_{\varphi_0}(w)| \geq |w|/2$ for $|w| \leq 1/2$.
If $\varphi_1 \in \Homeo(\S1)$ satisfies $\Vert \varphi_1-\varphi_0\Vert_{\S1} <\varepsilon$ with $0<\varepsilon \leq 1/12$, 
then $\xi_{\varphi_1}(w)$ has a zero, which is the barycenter of $\varphi_1$, in $|w|<6 \varepsilon$.
\end{lemma}

\begin{proof}
By $\Vert \varphi_1-\varphi_0\Vert_{\S1} <\varepsilon$, Proposition \ref{av-any} (1) gives
$|\xi_{\varphi_1}(w)-\xi_{\varphi_0}(w)|<3\varepsilon$ for $|w| \leq 1/2$.
On the other hand, on the circle $|w|=6\varepsilon \leq 1/2$, we have
$|\xi_{\varphi_0}(w)| \geq |w|/2=3\varepsilon$.
Then the argument principle 
yields that 
the rotation numbers for $\xi_{\varphi_0}$ and $\xi_{\varphi_1}$ 
regarded as vector fields are the same along the circle $|w|=6\varepsilon$.
Since $\xi_{\varphi_0}(w)$ has the unique zero in $|w|<6\varepsilon$, the Poincar\'e-Hopf theorem implies that
$\xi_{\varphi_1}(w)$ also has a zero in $|w|<6\varepsilon$.
\end{proof}
\medskip

\section{The proof of the main theorem}\label{6}
\newtheorem{Claim}{Claim}
\renewcommand{\theClaim}{\arabic{Claim}}

%

This section is entirely devoted to the proof of the main theorem in the form of
Theorem \ref{main2}. Actually, we first show that it can be reduced to Theorem \ref{main3} below.
Then we aim to prove this theorem by dividing the arguments into 
several claims.

Fix an arbitrary $\eta \in \S1$.
Let $r_\eta \in \Mob(\S1)$ be the rotation that sends $1$ to $\eta$.
By composing suitable rotations, we have
$$
r_{\psi \circ \varphi_0(\eta)}^{-1} \circ \psi \circ \varphi_0 \circ r_\eta=
(r_{\psi \circ \varphi_0(\eta)}^{-1} \circ \psi \circ r_{\varphi_0(\eta)})\circ (r_{\varphi_0(\eta)}^{-1} \circ \varphi_0 \circ r_\eta),
$$
and set $\varphi_0^\eta=r_{\varphi_0(\eta)}^{-1} \circ \varphi_0 \circ r_\eta$ and
$\psi^\eta=r_{\psi \circ \varphi_0(\eta)}^{-1} \circ \psi \circ r_{\varphi_0(\eta)}$. They both fix $1$.
Moreover, we can choose $u_{\psi,\eta} \in (-1,1)$ such that $\psi^\eta_0:=h_{u_{\psi,\eta}} \circ \psi^\eta$ satisfies
$(\psi^\eta_0)'_{\S1}(1)=1$. 
Note that $\psi^\eta_0(1)=1$ still holds.
Under these assumptions, we will prove the following.

\begin{theorem}\label{main3}
Suppose that $\psi^\eta_0(1)=\varphi_0^\eta(1)=1$ and $(\psi^\eta_0)'_{\S1}(1)=1$. Then there exist a constant
$t_0 \in [0,1)$ depending only on $\varphi_0$ and a constant
$\widetilde A>0$ such that
$$
\sup_{t \in [t_0,1), \ \eta \in \S1}\left(\frac{2}{1-t^2}\right)^{\alpha}
|\mu_{e(\psi^\eta_0 \circ \varphi_0^\eta)}(t)-\mu_{e(\varphi_0^\eta)}(t)| \leq 
\widetilde A\, \sup_{\eta \in \S1} c_\alpha(\psi^\eta_0)(1). 
$$
\end{theorem}


 
\smallskip
\noindent
{\it Theorem \ref{main3} $\Rightarrow$ Theorem \ref{main2}.}
If $\psi$ converge to $\id$ in $\Diff^{1+\alpha}(\S1)$ as assumed in Theorem \ref{main2}, then
the $\alpha$-H\"older constant $c_\alpha(\psi)(\eta)$ of $\psi$ at $\eta$ in particular
converge to $0$ uniformly with respect to $\eta \in \S1$. 
Since $c_\alpha(\psi)(\eta)=c_\alpha(\psi^\eta)(1)$, this also converge to $0$ uniformly.
It also follows from 
the convergence of the derivative of $\psi$ that
$\psi'_{\S1}(\eta)=(\psi^\eta)'_{\S1}(1)$ converge to $1$ uniformly.
This implies that $u_{\psi,\eta}$ converge to $0$ uniformly.
Therefore, $c_\alpha(\psi^\eta_0)(1)$ also converge to $0$ uniformly with respect to $\eta \in \S1$.

The conformal naturality implies that
$$
\mu_{e(\psi^\eta_0 \circ \varphi_0^\eta)}(t)=\mu_{e(\psi \circ \varphi_0)}(z); \quad 
\mu_{e(\varphi_0^\eta)}(t)=\mu_{e(\varphi_0)}(z),
$$
for $z=t\eta \in \D$
and then the conclusion of Theorem \ref{main3} shows that
$$
\sup_{t_0 \leq |z| <1}\left(\frac{2}{1-|z|^2}\right)^{\alpha}|\mu_{e(\psi \circ \varphi_0)}(z)-\mu_{e(\varphi_0)}(z)| \to 0.
$$
On the other hand, for $z \in \D$ with $|z|<t_0$, $\mu_{e(\psi \circ \varphi_0)}(z)$ converge to $\mu_{e(\varphi_0)}(z)$
uniformly as $\psi$ converge to $\id$ uniformly, which was proved in Douady and Earle \cite[Proposition 2]{DE}.
This proves Theorem \ref{main2}.
\qed
\medskip

We consider the conjugate 
$\varphi_t^\eta=h_t^{-1} \circ \varphi_0^\eta \circ h_t$ for $t \in (-1,1)$. 
Set $(\varphi_0^\eta)'_{\S1}(1)=\ell_\eta$ and take $h_{s_\eta}$ with
$(1-s_\eta)/(1+s_\eta)=\ell_\eta$.
Since $\ell_\eta=(\varphi_0)'_{\S1}(\eta)$, there exists some constant $L \geq 1$ depending only on $\varphi_0$ such that
$L^{-1} \leq \ell_\eta \leq L$ for every $\eta \in \S1$.
For a certain constant $\varepsilon_0 \in (0,2]$, which will be fixed later, choose $\delta_0>0$ as in Lemma \ref{akari2}. 
{\it Now we consider any $t>0$ with $0<1-t \leq \delta_0$.}

\begin{Claim}
Under the above assumption, we have
$$
\Vert h_{s_\eta}^{-1} \circ \varphi_t^\eta -{\rm id} \Vert_{\S1} \leq \pi L\varepsilon_0/2.
$$
Moreover, the barycenter $w_{t,\eta}$ of
$h_{s_\eta}^{-1} \circ \varphi_t^\eta$ satisfies
$|w_{t,\eta}| \leq \pi L\varepsilon_0$.
\end{Claim}

\begin{proof}
Lemma \ref{akari2} asserts that
if $1-t \leq \delta_0$ then
$|\varphi_t^\eta(\zeta)-h_{s_\eta}(\zeta)| \leq \varepsilon_0$
for every $\zeta \in \S1$ and for every $\eta \in \S1$.
This condition implies that
$d_{\S1}(\varphi_t^\eta(\zeta),h_{s_\eta}(\zeta)) \leq \pi \varepsilon_0/2$.
Since $|(h_{s_\eta}^{-1})'(\zeta)| 
\leq L$ by Proposition \ref{claims} (2)
applied to $w=s_\eta$, 
we have 
$d_{\S1}(h_{s_\eta}^{-1} \circ \varphi_t^\eta(\zeta),\zeta) \leq \pi L \varepsilon_0/2$ for every $\zeta \in \S1$.
This means the first statement.
Then Corollary \ref{barycenter} implies that $|w_{t,\eta}| \leq \pi L \varepsilon_0$.
\end{proof}

Using this barycenter $w_{t,\eta}$, we set
$$
j_{t,\eta}(z)=\frac{z-w_{t,\eta}}{1-\overline{w_{t,\eta}} z}.
$$
Furthermore, we define $g_{t,\eta}=j_{t,\eta} \circ h_{s_\eta}^{-1} \in \Mob(\D)$.
Then the constant $\varepsilon_0 \in (0,2]$ is given as follows. 
First, we prepare the following inequality:
\begin{align*}
\Vert g_{t,\eta} \circ \varphi_t^\eta-{\rm id}\Vert_{\S1}&=\Vert j_{t,\eta} \circ h_{s_\eta}^{-1} \circ \varphi_t^\eta-{\rm id} \Vert_{\S1}\\
&\leq \Vert j_{t,\eta} \circ h_{s_\eta}^{-1} \circ \varphi_t^\eta-h_{s_\eta}^{-1} \circ \varphi_t^\eta\Vert_{\S1}+
\Vert h_{s_\eta}^{-1} \circ \varphi_t^\eta-{\rm id}\Vert_{\S1}\\
&\leq \frac{2 \cdot \pi L \varepsilon_0}{1-\pi L \varepsilon_0}+\frac{\pi L \varepsilon_0}{2}.
\end{align*}
Here the last inequality is due to Proposition \ref{claims} (1) and Claim 1.
We set the last term in the above inequalities as $\tilde \varepsilon_0$.
{\it Now we choose $\varepsilon_0>0$ so that $0<\tilde \varepsilon_0 \leq 1/112$.}
This in particular gives $|w_{t,\eta}| \leq 1/280$ by Claim 1.

\begin{Claim}
The average function of $g_{t,\eta} \circ \varphi_t^\eta$ given by 
$$
\xi(w)=\frac{1}{2\pi} \int_{\S1} 
\frac{g_{t,\eta} \circ \varphi_t^\eta(\zeta)-w}{1-\bar w g_{t,\eta} \circ \varphi_t^\eta(\zeta)}\,|d\zeta|
$$
satisfies $\xi(0)=0$ and
$|\xi(w)| \geq |w|/2$
for $|w| \leq 1/2$. 
\end{Claim}

\begin{proof}
The barycenter of $g_{t,\eta} \circ \varphi_t^\eta$ is
$$
e(g_{t,\eta} \circ \varphi_t^\eta)(0)=j_{t,\eta}(e(h_{s_\eta}^{-1} \circ \varphi_t^\eta)(0))=j_{t,\eta}(w_{t,\eta})=0.
$$
This means that $\xi(0)=0$.
Then Lemma \ref{lower} with $\Vert g_{t,\eta} \circ \varphi_t^\eta-{\rm id}\Vert_{\S1} \leq \tilde \varepsilon_0$ implies that
$$
|\xi(w)| \geq (1-56 \tilde \varepsilon_0)|w| \geq |w|/2
$$
for $|w| \leq 1/2$. 
\end{proof}

For the same $t>0$ with $0<1-t \leq \delta_0$ as above, 
consider the conjugate $\psi_t^\eta=h_t^{-1} \circ \psi^\eta_0 \circ h_t$ and
the decomposition
$$
g_{t,\eta} \circ \psi_t^\eta \circ \varphi_t^\eta=
(g_{t,\eta} \circ \psi_t^\eta \circ  g_{t,\eta}^{-1}) \circ (g_{t,\eta} \circ \varphi_t^\eta).
$$
Since $|g_{t,\eta}(0)|=|g_{t,\eta}^{-1}(0)|=|h_{s_\eta}(w_{t,\eta})|$ and $|w_{t,\eta}| \leq 1/280$, 
there is $r \in [0,1)$ depending only on $L$
such that $|g_{t,\eta}(0)| \leq r$. Set $R=(1+r)/(1-r)$.
{\it Take $\varepsilon>0$ arbitrarily with $\varepsilon \leq 1/(42R)$, and assume hereafter that
$\Vert \psi_t^\eta-\id \Vert_{\S1} <\varepsilon$.}

\begin{Claim}
The barycenter $w_\varepsilon$ of $g_{t,\eta} \circ \psi_t^\eta \circ \varphi_t^\eta$
satisfies $|w_\varepsilon|<6R\varepsilon$.
\end{Claim}

\begin{proof}
Since $d_{\S1}(\psi_t^\eta(\zeta),\zeta)<\varepsilon$ and $|g_{t,\eta}(0)| \leq r$, we see from Proposition \ref{claims} (2) that
$$
d_{\S1}(g_{t,\eta} \circ \psi_t^\eta \circ  g_{t,\eta}^{-1}(\zeta),\zeta) < R \varepsilon
$$
for every $\zeta \in \S1$.
By replacing $\zeta$ with $g_{t,\eta} \circ \varphi_t^\eta (\zeta)$, we have
$$
\Vert g_{t,\eta} \circ \psi_t^\eta \circ  \varphi_t^\eta-g_{t,\eta} \circ \varphi_t^\eta \Vert_{\S1} < R \varepsilon 
\ (\leq 1/42).
$$
Since $g_{t,\eta} \circ \varphi_t^\eta$ satisfies the properties in Claim 2,
Lemma \ref{rouche} asserts that $g_{t,\eta} \circ \psi_t^\eta \circ \varphi_t^\eta$
has the barycenter in $|w| <6R \varepsilon$. 
\end{proof}

Using this barycenter $w_\varepsilon$, we set
$$
j_{\varepsilon}(z)=\frac{z-w_{\varepsilon}}{1-\overline{w_{\varepsilon}} z}.
$$
Furthermore, we define $g_{\varepsilon,t,\eta}=j_{\varepsilon} \circ g_{t,\eta} \in \Mob(\D)$.
Then the barycenter of $g_{\varepsilon,t,\eta} \circ \psi_t^\eta \circ \varphi_t^\eta$ is $0$. This is because
$$
e(g_{\varepsilon,t,\eta} \circ \psi_t^\eta \circ \varphi_t^\eta)(0)
=j_{\varepsilon}(e(g_{t,\eta} \circ \psi_t^\eta \circ \varphi_t^\eta)(0))=j_{\varepsilon}(w_{\varepsilon})=0.
$$

\begin{Claim}
$\Vert g_{\varepsilon,t,\eta} \circ \psi_t^\eta \circ \varphi_t^\eta- g_{t,\eta} \circ \varphi_t^\eta \Vert_{\S1} 
< 15R\varepsilon \leq 5/14$.
\end{Claim}

\begin{proof}
We have obtained
$d_{\S1}(g_{t,\eta} \circ \psi_t^\eta \circ  g_{t,\eta}^{-1}(\zeta),\zeta) < R \varepsilon$
for every $\zeta \in \S1$.
Then Proposition \ref{claims} (1) and Claim 3 yield
\begin{align*}
&\quad \ d_{\S1}(g_{\varepsilon,t,\eta} \circ \psi_t^\eta \circ \varphi_t^\eta(\zeta) ,g_{t,\eta} \circ \varphi_t^\eta(\zeta))\\
&= d_{\S1}(g_{\varepsilon,t,\eta} \circ \psi_t^\eta  \circ g_{t,\eta}^{-1}(\zeta),\zeta)\\
&\leq
d_{\S1}(j_\varepsilon \circ g_{t,\eta} \circ \psi_t^\eta \circ g_{t,\eta}^{-1}(\zeta),
g_{t,\eta} \circ \psi_t^\eta \circ g_{t,\eta}^{-1}(\zeta))+
d_{\S1}(g_{t,\eta} \circ \psi_t^\eta \circ g_{t,\eta}^{-1}(\zeta),\zeta)\\
&< \frac{2 \cdot 6R\varepsilon}{1-6R \varepsilon}+R\varepsilon.
\end{align*}
Since we have chosen $\varepsilon>0$ so that $\varepsilon \leq 1/(42R)$,
the last term in the above inequality is bounded by $15R\varepsilon \leq 5/14$.
\end{proof}

We will compute the complex dilatation of the conformally natural extensions of
$\varphi_t^\eta$ and $\psi_t^\eta \circ \varphi_t^\eta$ at $0 \in \D$ and estimate their difference.
For this purpose, we replace them with $g_{t,\eta} \circ \varphi_t^\eta$ and
$g_{\varepsilon,t,\eta} \circ \psi_t^\eta \circ \varphi_t^\eta$ respectively.
This is possible because the post composition of a M\"obius transformation does not affect the complex dilatation.
In addition, since the barycenters of both $g_{t,\eta} \circ \varphi_t^\eta$ and
$g_{\varepsilon,t,\eta} \circ \psi_t^\eta \circ \varphi_t^\eta$ are $0$ as we have seen above,
we can represent the complex dilatations explicitly in terms of the Fourier coefficients for
$g_{t,\eta} \circ \varphi_t^\eta$ and
$g_{\varepsilon,t,\eta} \circ \psi_t^\eta \circ \varphi_t^\eta$ as mentioned in Section 3. Namely,
$$
\mu_{e(\varphi_t^\eta)}(0)=\mu_{e(g_{t,\eta} \circ \varphi_t^\eta)}(0)=\frac{a_{-1}-\bar a_1b}{a_1-\bar a_{-1}b},
$$
where
$$
a_1=\frac{1}{2\pi}\int_{\S1} \bar \zeta (g_{t,\eta} \circ \varphi_t^\eta)(\zeta)\,|d\zeta|;\quad
a_{-1}=\frac{1}{2\pi}\int_{\S1} \zeta (g_{t,\eta} \circ \varphi_t^\eta)(\zeta)\,|d\zeta|;\quad
$$
$$
b=\frac{-1}{2\pi}\int_{\S1} (g_{t,\eta} \circ \varphi_t^\eta)(\zeta)^2\,|d\zeta|.
$$
Similarly,
$$
\mu_{e(\psi_t^\eta \circ \varphi_t^\eta)}(0)=\mu_{e(g_{\varepsilon,t,\eta} \circ \psi_t^\eta \circ \varphi_t^\eta)}(0)
=\frac{a'_{-1}-\bar a'_1b'}{a'_1-\bar a'_{-1}b'},
$$
where
$$
a'_1=\frac{1}{2\pi}\int_{\S1} \bar \zeta (g_{\varepsilon,t,\eta} \circ \psi_t^\eta \circ \varphi_t^\eta)(\zeta)\,|d\zeta|;\quad
a'_{-1}=\frac{1}{2\pi}\int_{\S1} \zeta (g_{\varepsilon,t,\eta} \circ \psi_t^\eta \circ \varphi_t^\eta)(\zeta)\,|d\zeta|;\quad
$$
$$
b'=\frac{-1}{2\pi}\int_{\S1} (g_{\varepsilon,t,\eta} \circ \psi_t^\eta \circ \varphi_t^\eta)(\zeta)^2\,|d\zeta|.
$$

In Claim 4, we have obtained the difference between $g_{t,\eta} \circ \varphi_t^\eta$ and
$g_{\varepsilon,t,\eta} \circ \psi_t^\eta \circ \varphi_t^\eta$.
Hence it follows that
\begin{align*}
|a_1-a'_1| &\leq \frac{1}{2\pi} \int_{\S1} |\bar \zeta| 15R\varepsilon\,|d\zeta|=15R\varepsilon;\\
|a_{-1}-a'_{-1}| &\leq \frac{1}{2\pi} \int_{\S1} |\zeta| 15R\varepsilon\,|d\zeta|=15R\varepsilon;\\
|b-b'|&\leq \frac{1}{2\pi} \int_{\S1} 2\cdot 15R\varepsilon\,|d\zeta|=30R\varepsilon.
\end{align*}
On the other hand,
$$
|\mu_{e(\varphi_t^\eta)}(0)-\mu_{e(\psi_t^\eta \circ \varphi_t^\eta)}(0)|=
\left|\frac{a_{-1}-\bar a_1b}{a_1-\bar a_{-1}b}-\frac{a'_{-1}-\bar a'_1b'}{a'_1-\bar a'_{-1}b'}\right |
=:\frac{N}{|a_1-\bar a_{-1}b| \cdot |a'_1-\bar a'_{-1}b'|}.
$$
Here, simple computation and the above inequalities show that
the numerator $N$ is estimated from above by a positive constant multiple of $\varepsilon$.

For the estimate of the denominator from below, we first consider the following:
\begin{align*}
|a_1-\bar a_{-1}b| &\geq |a_1|-|a_{-1}||b| \geq  |a_1|-|a_{-1}|;\\
|a'_1-\bar a'_{-1}b'| &\geq |a'_1|-|a'_{-1}||b'| \geq  |a'_1|-|a'_{-1}|.
\end{align*}
We set
$\delta=|a_1|^2-|a_{-1}|^2$ and $\delta'=|a'_1|^2-|a'_{-1}|^2$.
Then
$$
|a_1|-|a_{-1}| =\frac{\delta}{|a_1|+|a_{-1}|} \geq \frac{\delta}{2}; 
\quad |a'_1|-|a'_{-1}| =\frac{\delta}{|a'_1|+|a'_{-1}|} \geq \frac{\delta'}{2}.
$$
Here, we see that $g_{t,\eta} \circ \varphi_t^\eta$ and $g_{\varepsilon,t,\eta} \circ \psi_t^\eta \circ \varphi_t^\eta$
are uniformly close to $\id$ within $\pi/6$. Indeed, the definitions of $\tilde \varepsilon_0$ and Claim 4 give that
\begin{align*}
& \Vert g_{t,\eta} \circ \varphi_t^\eta-\id \Vert_{\S1} \leq \tilde \varepsilon_0 \leq 1/112;\\
& \Vert g_{\varepsilon,t,\eta} \circ \psi_t^\eta \circ \varphi_t^\eta-g_{t,\eta} \circ \varphi_t^\eta \Vert_{\S1}
\leq 15R\varepsilon \leq 5/14.
\end{align*}
Then, by Pommerenke \cite[Lemma 5.18]{P} interpreting \cite[Lemma 3]{DE}, we have that
$\delta$ and $\delta'$ are uniformly bounded away from $0$.
Thus, we can find some absolute constant $A>0$ such that
$$
|\mu_{e(\varphi_t^\eta)}(0)-\mu_{e(\psi_t^\eta \circ \varphi_t^\eta)}(0)| \leq A\varepsilon
$$
for every $\eta \in \S1$ and every $t \in [1-\delta_0,1)$. 

The conformal naturality again yields
\begin{align*}
& \mu_{e(\varphi_t^\eta)}(0)=\mu_{e(h_t^{-1} \circ \varphi_0^\eta \circ h_t)}(0)
=\mu_{e(h_t^{-1} \circ \varphi_0^\eta)}(h_t(0)) \frac{\overline{h'_t(0)}}{h'_t(0)}
=\mu_{e(\varphi_0^\eta)}(t) \frac{\overline{h'_t(0)}}{h'_t(0)};\\
& \mu_{e(\psi_t^\eta \circ \varphi_t^\eta)}(0)=\mu_{e(h_t^{-1} \circ \psi^\eta_0 \circ \varphi_0^\eta \circ h_t)}(0)
=\mu_{e(h_t^{-1} \circ \psi^\eta_0 \circ \varphi_0^\eta)}(h_t(0)) \frac{\overline{h'_t(0)}}{h'_t(0)}
=\mu_{e(\psi^\eta_0 \circ \varphi_0^\eta)}(t) \frac{\overline{h'_t(0)}}{h'_t(0)}.
\end{align*}
Therefore,
$$
|\mu_{e(\varphi_0^\eta)}(t)-\mu_{e(\psi^\eta_0 \circ \varphi_0^\eta)}(t)|=
|\mu_{e(\varphi_t^\eta)}(0)-\mu_{e(\psi_t^\eta \circ \varphi_t^\eta)}(0)| \leq A\varepsilon
$$
for every $\eta \in \S1$ and every $t \in [1-\delta_0,1)$. 

The assumption for this conclusion was that
$\Vert \psi_t^\eta-\id \Vert_{\S1} <\varepsilon$ for $\varepsilon \leq 1/(42R)$.
Proposition \ref{setup} and
Lemma \ref{key} tell us that if we choose $t$ and $\epsilon:=2\varepsilon/\pi$ in the relation
$1-t = \frac{1}{4}(\epsilon/(4C))^{1/\alpha}$, then
we have that condition. Here, $C=C_\eta$ is written by the $\alpha$-H\"older constant $c_\alpha(\psi^\eta_0)(1)$ for 
$\psi^\eta_0$ at $1$ as
$$
C_\eta=\frac{c_\alpha(\psi^\eta_0)(1)(\pi/2)^{\alpha+1}}{\alpha+1}.
$$
This relation can be alternatively written as
$$
\epsilon=4^{\alpha+1}C_\eta(1-t)^\alpha \leq \frac{1}{21\pi R}.
$$
We may assume that $C_\eta$ are uniformly bounded by some fixed positive constant, say, one. 
Then we can find a constant $t_0$
with $1-\delta_0 \leq t_0 <1$ depending only on $R$, and hence only on $\varphi_0$,
such that 
$$
|\mu_{e(\psi^\eta_0 \circ \varphi_0^\eta)}(t)-\mu_{e(\varphi_0^\eta)}(t)| \leq A \cdot \frac{4^{\alpha+1}}{\alpha+1}
(\pi/2)^{\alpha+2}c_\alpha(\psi^\eta_0)(1)(1-t)^\alpha
$$
for every $\eta \in \S1$ and every $t \in [t_0,1)$. 
Taking the supremum over $\eta \in \S1$ and $t \in [t_0,1)$, we have
$$
\sup_{t \in [t_0,1), \ \eta \in \S1}\left(\frac{2}{1-t^2}\right)^{\alpha}
|\mu_{e(\psi^\eta_0 \circ \varphi_0^\eta)}(t)-\mu_{e(\varphi_0^\eta)}(t)| \leq 
\widetilde A\, \sup_{\eta \in \S1} c_\alpha(\psi^\eta_0)(1)
$$
for some constant $\widetilde A>0$.
This completes the proof of Theorem \ref{main3}.

\bigskip

\end{document}